\documentclass[12pt]{article}
\usepackage{fullpage}
\usepackage{amssymb,amsmath,amsfonts,amsthm,latexsym,enumerate,url,cases}
\numberwithin{equation}{section}
\usepackage{hyperref}
\hypersetup{colorlinks=true,citecolor=blue,linkcolor=blue,urlcolor=blue}

\newtheorem{theorem}{Theorem}[section] %
\newtheorem{lemma}[theorem]{Lemma} %
\newtheorem{corollary}[theorem]{Corollary} %
\newtheorem{problem}[theorem]{Problem} %
\newtheorem{remark}[theorem]{Remark} %
\usepackage[left=50pt,right=50pt,top=30pt,bottom=80pt]{geometry}
\begin{document}
\title{On the values of representation functions II}
\author{{Xing-Wang Jiang\footnote{xwjiangnj@sina.com(X.-W. Jiang)}, Csaba S\'{a}ndor\footnote{csandor@math.bme.hu(C. S\'{a}ndor)}, Quan-Hui Yang\footnote{yangquanhui01@163.com(Q.-H. Yang)} }\\
\small  *School of Mathematical Sciences and Institute of Mathematics,\\\small Nanjing Normal University, Nanjing 210023, China\\
\small  $\dagger$Institute of Mathematics, Budapest University of Technology and Economics,\\
\small  H-1529 B. O. Box, Hungary\\
\small  $\ddagger$School of Mathematics and Statistics, Nanjing University of Information \\\small Science and Technology,
Nanjing 210044, China
}
\date{}
\maketitle \baselineskip 18pt \maketitle \baselineskip 18pt

{\bf Abstract.}
For a set $A$ of nonnegative integers, let $R_2(A,n)$ and $R_3(A,n)$ denote the number of solutions to $n=a+a'$ with $a,a'\in A$, $a<a'$ and $a\leq a'$, respectively. In this paper, we prove that, if $A\subseteq \mathbb{N}$ and $N$ is a positive integer such that $R_2(A,n)=R_2(\mathbb{N}\setminus A,n)$ for all $n\geq2N-1$, then for any $\theta$ with $0<\theta<\frac{2\log2-\log3}{42\log 2-9\log3}$, the set of integers $n$ with $R_2(A,n)=\frac{n}{8}+O(n^{1-\theta})$ has density one. The similar result holds for $R_3(A,n)$. These improve the results of the first author.


\vskip 5mm

\section{Introduction}
Let $\mathbb{N}$ be the set of nonnegative integers. For a set $A\subseteq\mathbb{N}$, let $R_2(A,n)$ and $R_3(A,n)$ denote the number of solutions to $a+a'=n$ with $a,a'\in A$, $a<a'$ and $a\leq a'$, respectively.
In the last few decades, the representation function was a popular topic that was studied by Erd\H{o}s, S\'{a}rk\"{o}zy and S\'{o}s in a series of papers, see \cite{ES1,ES2,ESS1}. S\'{a}rk\"{o}zy asked that whether there exist two sets $A$ and $B$ with infinite symmetric difference such that $R_i(A,n)=R_i(B,n)$ for all sufficiently large integers $n$. Let $A_0$ be the set of all nonnegative integers $n$ with even number of ones in the binary representation of $n$, and let $B_0=\mathbb{N}\setminus A_0$. The sequence $A_0$ is called Thue-Morse sequence. In 2002, Dombi \cite{D} answered this problem affirmatively by proving that $R_2(A_0,n)=R_2(B_0,n)$ for all $n\geq0$. In 2003, Chen and Wang \cite{CW} proved that the set of nonnegative integers can be partitioned into two subsets $A$ and $B$ such that $R_3(A,n)=R_3(B,n)$ for all $n\geq1$. In 2004, S\'{a}ndor \cite{S} gave the precise formulations as following.
\vskip 2mm
\noindent\text{Theorem A} \cite[Theorem 1]{S}.
{\it Let $N$ be a positive integer. Then $R_2(A,n)=R_2(\mathbb{N}\setminus A,n)$ for all $n\geq 2N-1$ if and only if $|A\cap[0,2N-1]|=N$ and $2m\in A\Leftrightarrow m\in A,~2m+1\in A\Leftrightarrow m\notin A$ for all $m\geq N$.}
\vskip 2mm
\noindent\text{Theorem B} \cite[Theorem 2]{S}.
{\it Let $N$ be a positive integer. Then $R_3(A,n)=R_3(\mathbb{N}\setminus A,n)$ for all $n\geq 2N-1$ if and only if $|A\cap[0,2N-1]|=N$ and $2m\in A\Leftrightarrow m\notin A,~2m+1\in A\Leftrightarrow m\in A$ for all $m\geq N$.}
\vskip 2mm
Let $R_{A,B}(n)$ be the number of solutions to
$a+b=n,a\in A,b\in B.$ In 2011, Chen \cite{C} studied the range of $R_i(A,n)$ for the first time and proved the
following results.
\vskip 2mm
\noindent\text{Theorem C} \cite[Theorem 1.1, Theorem 1.4 (i)]{C}.
{\it (i) Let $A$ be a subset of $\mathbb{N}$ and $N$ be a positive integer such that $R_2(A,n)=R_2(\mathbb{N}\setminus A,n)$ for all $n\geq 2N-1$. If $|A\cap A_0|=+\infty$ and $|A\cap B_0|=+\infty$, then for all $n\geq1$, we have
$$R_2(A,n)\geq\frac{n}{40N(N+1)}-1,~~R_{A,\mathbb{N}\setminus A}(n)\geq \frac{n}{20N(N+1)}-1.$$
(ii) Let $A$ be a subset of $\mathbb{N}$ and $N$ be a positive integer such that $R_3(A,n)=R_3(\mathbb{N}\setminus A,n)$ for all $n\geq 2N-1$. Then for all $n\geq1$, we have
$$R_3(A,n)\geq\frac{n}{60(N+1)}-\frac{2N}{3},~~R_{A,\mathbb{N}\setminus A}(n)\geq \frac{n}{30(N+1)}-\frac{4N}{3}.$$}

\vskip -4mm
\noindent\text{Theorem D} (\cite[Theorem 1.2, Theorem 1.5]{C}).
{\it (i) Let $A$ be a subset of $\mathbb{N}$ and $N$ be a positive integer such that $R_2(A,n)=R_2(\mathbb{N}\setminus A,n)$ for all $n\geq 2N-1$. Then, for any function $f$ with $f(x)\rightarrow +\infty$ as $x\rightarrow +\infty$, the set of integers $n$ with
$$R_2(A,n)\geq \frac{n}{16}-f(n),~R_{A,\mathbb{N}\setminus A}(n)\geq \frac{n}{8}-f(n)$$
has the density one.\\
(ii) Let $A$ be a subset of $\mathbb{N}$ and $N$ be a positive integer such that $R_3(A,n)=R_3(\mathbb{N}\setminus A,n)$ for all $n\geq 2N-1$. Then, for any function $f$ with $f(x)\rightarrow +\infty$ as $x\rightarrow +\infty$, the set of integers $n$ with
$$R_3(A,n)\geq \frac{n}{16}-f(n),~R_{A,\mathbb{N}\setminus A}(n)\geq \frac{n}{8}-f(n)$$
has the density one.}
\vskip 2mm
In 2018, Jiang \cite{J} improved the results of Theorem D.
\vskip 2mm
\noindent\text{Theorem E} \cite[Corollary 1.2]{J}.
{\it Let $A$ be a subset of $\mathbb{N}$ and $N$ be a positive integer such that $R_2(A,n)=R_2(\mathbb{N}\setminus A,n)$ for all $n\geq 2N-1$. Then, for any $\epsilon>0$, the set of integers $n$ with
$$\left(\frac{1}{8}-\epsilon\right)n\leq R_2(A,n)\leq\left(\frac{1}{8}+\epsilon\right)n$$
has density one.}
\vskip 2mm
Jiang claimed that if $A\subseteq \mathbb{N}$ and $N$ is a positive integer such that $R_3(A,n)=R_3(\mathbb{N}\setminus A,n)$ for all $n\geq 2N-1$, then the same result holds for $R_3(A,n)$. For more
related results, see \cite{CL,CSST,T,TY,YT}. In this paper, the following results are proved.

{\bf
\begin{theorem}\label{thm0}
Let $A_0$ be the Thue Morse sequence. Then for any $0<\theta<\frac{2\log2-\log3}{40\log 2-8\log3} =0.0151\dots $, there exists a $\delta=\delta(\theta)>0$ such that
$$|\{n:n\leq x,R_2(A_0,n)=\frac{n}{8}+O(n^{1-\theta})\}|=x+O(x^{1-\delta}),$$
where the implied constant depends on $\theta$.
\end{theorem}}

\begin{theorem}\label{thm1}
Let $A$ be a subset of $\mathbb{N}$ and $N$ be a positive integer such that $R_2(A,n)=R_2(\mathbb{N}\setminus A,n)$ for all $n\geq 2N-1$. Then for any $0<\theta<\frac{2\log2-\log3}{42\log 2-9\log3}=0.0149\dots$, there exists a $\delta=\delta(\theta)>0$ such that
$$|\{n:n\leq x,R_2(A,n)=\frac{n}{8}+O(n^{1-\theta})\}|=x+O(x^{1-\delta}),$$
where the implied constant depends only on $\theta$.
\end{theorem}
By Theorem \ref{thm1}, we immediately have
\begin{corollary}
Let $A$ be a subset of $\mathbb{N}$ and $N$ be a positive integer such that $R_2(A,n)=R_2(\mathbb{N}\setminus A,n)$ for all $n\geq 2N-1$. Then for any $0<\theta<\frac{2\log2-\log3}{42\log 2-9\log3}=0.0149\dots$, the set of integers $n$ with
$$R_2(A,n)=\frac{n}{8}+O(n^{1-\theta})$$
has density one.
\end{corollary}
\begin{remark}
If $A\subseteq\mathbb{N}$ and $N$ is a positive integer such that $R_3(A,n)=R_3(\mathbb{N}\setminus A,n)$ for all $n\geq 2N-1$, then applying the same method, we can get that for any $0<\theta<\frac{2\log2-\log3}{42\log 2-9\log3}=0.0149\dots$, the set of integers $n$ with
$R_3(A,n)=\frac{n}{8}+O(n^{1-\theta})$
has density one.
\end{remark}
Motivated by $R_2(A_0,2^{2l+1}-1)=0$, we pose two problems for further research.
\begin{problem}
Let $A$ be a subset of $\mathbb{N}$ and $N$ be a positive integer such that $R_2(A,n)=R_2(\mathbb{N}\setminus A,n)$ for all $n\geq 2N-1$. Does there exist a sequence $n_1,n_2,\ldots$ such that
$$\lim_{k\rightarrow\infty}\frac{R_2(A,n_k)}{n_k}\neq\frac 1 8?$$
\end{problem}
\begin{problem}
Let $A$ be a subset of $\mathbb{N}$ and $N$ be a positive integer such that $R_3(A,n)=R_3(\mathbb{N}\setminus A,n)$ for all $n\geq 2N-1$. Does there exist a sequence $n_1,n_2,\ldots$ such that
$$\lim_{k\rightarrow\infty}\frac{R_3(A,n_k)}{n_k}\neq\frac 1 8?$$
\end{problem}

In this paper, we define
$$P_t=\{z:(\varepsilon_{3t}^{(z)},\varepsilon_{3t+1}^{(z)},\varepsilon_{3t+2}^{(z)})=(0,0,1)~\text{or}~(0,1,0)\},$$
and
$$Q_t=\left\{(y,z):\sum_{i=3t+3}^{\infty}\varepsilon_i^{(y)}2^i<\sum_{i=3t+3}^{\infty}\varepsilon_i^{(z)}2^i\right\},$$
where
$$y=\sum_{i=0}^{\infty}\varepsilon_i^{(y)}2^i,~~z=\sum_{i=0}^{\infty}\varepsilon_i^{(z)}2^i,~~\varepsilon_i^{(y)},\varepsilon_i^{(z)}\in\{0,1\}.$$

\section{Proofs}
In this section, we give the proof of our main result. Firstly, we prove some lemmas.

\begin{lemma}\label{lem1}
Let $0<\epsilon<\frac{2\log2-\log3}{40\log 2-8\log3}$. Then there exists a $\delta'=\delta'(\epsilon)$ with $0<\delta'<1$ such that
$$|\{n:\frac{x}{10}<n\leq x,R_2(A_0,n)=\frac{n}{8}+O(n^{1-\epsilon})\}|=\frac{9x}{10}+O(x^{1-\delta'}),$$
where the implied constant depends only on $\epsilon$.
\end{lemma}

\begin{proof}
For $\frac{x}{10}<n\leq x$, let
$$n=\sum_{i=0}^{\lfloor\log_2x\rfloor}\varepsilon_i^{(n)}2^i,~~\varepsilon_i^{(n)}\in\{0,1\}$$
be the binary representation of $n$, and let
$$f(n)=|\{k:(\varepsilon_{3k}^{(n)},\varepsilon_{3k+1}^{(n)},\varepsilon_{3k+2}^{(n)})=(1,0,1)\}|.$$
For $0<c<1$, we are going to show that $f(n)>\frac{c\log_2x}{24}$ implies $$R_2(A_0,n)=\frac n 8+O(n^{1-\frac{c(2\log2-\log3)}{40\log2-8\log3}}).$$
Suppose that $f(n)>\frac{c\log_2x}{24}$. Let $(\varepsilon_{3k_j}^{(n)},\varepsilon_{3k_j+1}^{(n)},\varepsilon_{3k_j+2}^{(n)})=(1,0,1)$ for $1\leq j\leq f(n)$, and let
\begin{eqnarray*}
S(n)&=&\{(y,z):n=y+z,0\leq y<z,z\in P_{k_t}~\text{but}~z\notin \cup_{j=1}^{t-1}P_{k_j},\\
&&(y,z)\in Q_{k_t}~\text{for~some}~t\in[1,f(n)]\}.
\end{eqnarray*}
Then
\begin{equation}\label{eq1.1}
|S(n)|=|S_1|-|S_2|-|S_3|,
\end{equation}
where
\begin{eqnarray*}
S_1&=&\{(y,z):n=y+z,0\leq y<z\},\\
S_2&=&\{(y,z)\in S_1:z\notin P_{k_t}~\text{for}~1\leq t\leq f(n)\},\\
S_3&=&\{(y,z)\in S_1:z\in P_{k_t}~\text{but}~z\notin \cup_{j=1}^{t-1}P_{k_j},(y,z)\notin Q_{k_t}~\text{for~some}~t\in[1,f(n)]\}.
\end{eqnarray*}
Clearly,
\begin{equation}\label{eq1.2}|
S_1|=\lfloor\frac{n+1}{2}\rfloor
\end{equation}
and
\begin{eqnarray}\label{eq1.3}
|S_2|=O(6^{f(n)}2^{\lfloor\log_2x\rfloor-3f(n)})=O\left(x\left(\frac{3}{4}\right)^{f(n)}\right)=O(x^{1-\frac{c(2\log2-\log3)}{24\log2}}).
\end{eqnarray}
We have
\begin{eqnarray*}
|S_3|=\sum_{t=1}^{f(n)}|\{(y,z)\in S_1:z\in P_{k_t}~\text{but}~z\notin \cup_{j=1}^{t-1}P_{k_j},~(y,z)\notin Q_{k_t}\}|.
\end{eqnarray*}
It follows from $(y,z)\in S_1$ and $(y,z)\notin P_{k_t}$
that
$$\sum_{i=3k_t+3}^{\lfloor \log_2x\rfloor}\varepsilon_{i}^{(y)}2^i=\sum_{i=3k_t+3}^{\lfloor \log_2x\rfloor}\varepsilon_{i}^{(z)}2^i=\lfloor\frac{n}{2^{3k_t+4}}\rfloor2^{3k_t+3},$$
therefore the digits $\varepsilon_{3k_t+3}^{(z)},\varepsilon_{3k_t+4}^{(z)},\ldots$ are determined. Since
$$3k_t\leq \log_2x-3(f(n)-t)\leq \log_2x-\frac{c\log_2x}{8}+3t,$$
it follows that
\begin{eqnarray*}
|\{(y,z)\in S_1:z\in P_{k_t}~\text{but}~z\notin\cup_{j=1}^{t-1}P_{k_j},~(y,z)\notin Q_{k_t}\}|=O(2^{\log_2x-\frac{c\log_2x}{8}+3t})=O(x^{1-\frac{c}{8}}2^{3t}),
\end{eqnarray*}
and so
\begin{eqnarray*}
&&\sum_{t\leq\frac{c\log2\log_2x}{40\log2-8\log3}}|\{(y,z)\in S_1:z\in P_{k_t}~\text{but}~z\notin\cup_{j=1}^{t-1}P_{k_j},~(y,z)\notin Q_{k_t}\}|\\
&=&\sum_{t\leq\frac{c\log2\log_2x}{40\log2-8\log3}}O(x^{1-\frac{c}{8}}2^{3t})=O(x^{1-\frac{c(2\log2-\log3)}{40\log2-8\log3}}).
\end{eqnarray*}
On the other hand,
\begin{eqnarray*}
&&|\{(y,z)\in S_1:z\in P_{k_t}~\text{but}~z\notin\cup_{j=1}^{t-1}P_{k_j},~(y,z)\notin Q_{k_t}\}|
\leq |\{(y,z)\in S_1:z\notin\cup_{j=1}^{t-1}P_{k_j}\}|\\
&=&O(6^{t-1}2^{\log_2x-3(t-1)})=O\left(x\left(\frac{3}{4}\right)^t\right),
\end{eqnarray*}
hence
\begin{eqnarray*}
&&\sum_{t\geq\frac{c\log2\log_2x}{40\log2-8\log3}}|\{(y,z)\in S_1:z\in P_{k_t}~\text{but}~z\notin\cup_{j=1}^{t-1}P_{k_j},~(y,z)\notin Q_{k_t}\}|\\
&&=\sum_{t\geq\frac{c\log2\log_2x}{40\log2-8\log3}}O\left(x\left(\frac{3}{4}\right)^t\right)=O(x^{1-\frac{c(2\log2-\log3)}{40\log2-8\log3}}).
\end{eqnarray*}
Therefore,
\begin{equation}\label{eq1.4}
|S_3|=O(x^{1-\frac{c(2\log2-\log3)}{40\log2-8\log3}}).
\end{equation}
By (\ref{eq1.1}), (\ref{eq1.2}), (\ref{eq1.3}), (\ref{eq1.4}), we have
$$|S(n)|=\frac{n}{2}+O(x^{1-\frac{c(2\log2-\log3)}{40\log2-8\log3}})=\frac{n}{2}+O(n^{1-\frac{c(2\log2-\log3)}{40\log2-8\log3}}).$$

Let
\begin{eqnarray*}
T(n)&=&\{(y,z):~n=y+z,~0\leq y<z,~\exists~L~s.t.~\sum_{i=L+3}^{\lfloor\log_2x\rfloor}\varepsilon_i^{(y)}2^i<\sum_{i=L+3}^{\lfloor\log_2x\rfloor}\varepsilon_i^{(z)}2^i,\\
&&(\varepsilon_{L+1}^{(y)},\varepsilon_{L+2}^{(y)},\varepsilon_{L}^{(z)},\varepsilon_{L+1}^{(z)},\varepsilon_{L+2}^{(z)})
=(0,0,0,0,1)~\text{or}~(1,0,0,1,0),\\
&&(\varepsilon_{i+1}^{(y)},\varepsilon_{i+2}^{(y)},\varepsilon_{i}^{(z)},\varepsilon_{i+1}^{(z)},\varepsilon_{i+2}^{(z)})
\neq(0,0,0,0,1)~\text{and}~(1,0,0,1,0)~\text{for}~i<L\}.
\end{eqnarray*}
We will prove that $S(n)\subseteq T(n)$.

Let $(y,z)\in S(n)$ and $(\varepsilon_{3k_t}^{(z)},\varepsilon_{3k_t+1}^{(z)},\varepsilon_{3k_t+2}^{(z)})=(0,1,0)$, then
$$\sum_{i=0}^{3k_t}\varepsilon_i^{(y)}2^i+\sum_{i=0}^{3k_t}\varepsilon_i^{(z)}2^i=\sum_{i=0}^{3k_t}\varepsilon_i^{(n)}2^i,$$
since $\sum_{i=0}^{3k_t}\varepsilon_i^{(y)}2^i+\sum_{i=0}^{3k_t}\varepsilon_i^{(z)}2^i<2^{3k_t}+2^{3k_t+1}$ and $\varepsilon_{3k_t}^{(n)}=1$. It follows from $(\varepsilon_{3k_t}^{(n)},\varepsilon_{3k_t+1}^{(n)},\varepsilon_{3k_t+2}^{(n)})=(1,0,1)$ that $\varepsilon_{3k_t+1}^{(y)}=1,~\varepsilon_{3k_t+2}^{(y)}=0$.
Hence, $$(\varepsilon_{3k_t+1}^{(y)},\varepsilon_{3k_t+2}^{(y)},\varepsilon_{3k_t}^{(z)},\varepsilon_{3k_t+1}^{(z)},\varepsilon_{3k_t+1}^{(z)})=(1,0,0,1,0),$$
and so there exists some $L\leq 3k_t$ satisfying
$$(\varepsilon_{L+1}^{(y)},\varepsilon_{L+2}^{(y)},\varepsilon_{L}^{(z)},\varepsilon_{L+1}^{(z)},\varepsilon_{L+2}^{(z)})
=(0,0,0,0,1) ~\text{or} ~(1,0,0,1,0),$$ but $$(\varepsilon_{i+1}^{(y)},\varepsilon_{i+2}^{(y)},\varepsilon_{i}^{(z)},\varepsilon_{i+1}^{(z)},\varepsilon_{i+2}^{(z)})
\neq(0,0,0,0,1)~\text{and}~(1,0,0,1,0)~\text{for}~i<L.$$ Since $\sum_{i=3k_t+3}^{\lfloor\log_2x\rfloor}\varepsilon_i^{(y)}2^i<\sum_{i=3k_t+3}^{\lfloor\log_2x\rfloor}\varepsilon_i^{(z)}2^i$, we have $\sum_{i=L+3}^{\lfloor\log_2x\rfloor}\varepsilon_i^{(y)}2^i<\sum_{i=L+3}^{\lfloor\log_2x\rfloor}\varepsilon_i^{(z)}2^i$.
Therefore, $(y,z)\in T(n)$.

Similarly, if $(y,z)\in S(n)$ and $(\varepsilon_{3k_t}^{(z)},\varepsilon_{3k_t+1}^{(z)},\varepsilon_{3k_t+2}^{(z)})=(0,0,1)$, then we obtain $\varepsilon_{3k_t+1}^{(y)}=0,~\varepsilon_{3k_t+2}^{(y)}=0$. By the similar argument as above, we have $(y,z)\in T(n)$. Therefore $S(n)\subseteq T(n)$, and so
\begin{equation}\label{eq1.5}
|T(n)|=\frac{n}{2}+O(n^{1-\frac{c(2\log2-\log3)}{40\log2-8\log3}}).
\end{equation}
Now, we prove that
$$R_2(A_0,n)=\frac{n}{8}+O(n^{1-\frac{c(2\log2-\log3)}{40\log2-8\log3}}).$$

For $D_1,D_2\subseteq \mathbb{N}$, let
$$T_{D_1,D_2}(n)=\{(y,z):(y,z)\in T(n),y\in D_1,z\in D_2\}.$$
Then
\begin{eqnarray}\label{eq1.6}
T(n)=T_{A_0,A_0}(n)\cup T_{B_0,B_0}(n)\cup T_{A_0,B_0}(n)\cup T_{B_0,A_0}(n).
\end{eqnarray}

We are going to define a bijection between $T_{A_0,A_0}(n)\cup T_{B_0,B_0}(n)$ and $T_{A_0,B_0}(n)\cup T_{B_0,A_0}(n)$. Let $(y,z)\in T(n)$. Let $(y',z')=(y-2^{L+1},z+2^{L+1})$ if $(\varepsilon_{L+1}^{(y)},\varepsilon_{L+2}^{(y)},\varepsilon_{L}^{(z)},\varepsilon_{L+1}^{(z)},\varepsilon_{L+2}^{(z)})=(1,0,0,1,0)$. It is easy to see that $(\varepsilon_{L+1}^{(y')},\varepsilon_{L+2}^{(y')},\varepsilon_{L}^{(z')},\varepsilon_{L+1}^{(z')},\varepsilon_{L+2}^{(z')})=(0,0,0,0,1)$, $\sum_{i=0}^{\infty}\varepsilon_i^{(y)}+\sum_{i=0}^{\infty}\varepsilon_i^{(z)}-(\sum_{i=0}^{\infty}\varepsilon_i^{(y')}+\sum_{i=0}^{\infty}\varepsilon_i^{(z')})=1$ and $(\varepsilon_{i+1}^{(y')},\varepsilon_{i+2}^{(y')},\varepsilon_{i}^{(z')},\varepsilon_{i+1}^{(z')},\varepsilon_{i+2}^{(z')})\ne (1,0,0,1,0)$ and $(0,0,0,0,1)$ for $i<L$.

Let $(y',z')=(y+2^{L+1},z-2^{L+1})$ if $(\varepsilon_{L+1}^{(y)},\varepsilon_{L+2}^{(y)},\varepsilon_{L}^{(z)},\varepsilon_{L+1}^{(z)},\varepsilon_{L+2}^{(z)})=(0,0,0,0,1)$. It is easy to see that $(\varepsilon_{L+1}^{(y')},\varepsilon_{L+2}^{(y')},\varepsilon_{L}^{(z')},\varepsilon_{L+1}^{(z')},\varepsilon_{L+2}^{(z')})=(1,0,0,1,0)$, $\sum_{i=0}^{\infty}\varepsilon_i^{(y)}+\sum_{i=0}^{\infty}\varepsilon_i^{(z)}-(\sum_{i=0}^{\infty}\varepsilon_i^{(y')}+\sum_{i=0}^{\infty}\varepsilon_i^{(z')})=-1$ and $(\varepsilon_{i+1}^{(y')},\varepsilon_{i+2}^{(y')},\varepsilon_{i}^{(z')},\varepsilon_{i+1}^{(z')},\varepsilon_{i+2}^{(z')})\ne (1,0,0,1,0)$ and $(0,0,0,0,1)$ for $i<L$.

Clearly, $\{ (y,z): (y,z)\in T(n), \sum_{i=0}^{\infty}\varepsilon_i^{(y)}+\sum_{i=0}^{\infty}\varepsilon_i^{(z)}\mbox{ is even} \}=T_{A_0,A_0}(n)\cup T_{B_0,B_0}(n)$ and $\{ (y,z): (y,z)\in T(n), \sum_{i=0}^{\infty}\varepsilon_i^{(y)}+\sum_{i=0}^{\infty}\varepsilon_i^{(z)}\mbox{ is odd} \}=T_{A_0,B_0}(n)\cup T_{B_0,A_0}(n)$.

These facts imply that $(y,z)\rightarrow(y',z')$ defined as above is a bijection between $T_{A_0,A_0}(n)\cup T_{B_0,B_0}(n)$ and $T_{A_0,B_0}(n)\cup T_{B_0,A_0}(n)$.

By (\ref{eq1.5}) and (\ref{eq1.6}), we have
$$|T_{A_0,A_0}(n)|+|T_{B_0,B_0}(n)|=\frac{n}{4}+O(n^{1-\frac{c(2\log2-\log3)}{40\log2-8\log3}}).$$
Clearly,
$$|T_{A_0,A_0}(n)|=R_2(A_0,n)+O(n^{1-\frac{c(2\log2-\log3)}{40\log2-8\log3}}),$$
and
$$|T_{B_0,B_0}(n)|=R_2(B_0,n)+O(n^{1-\frac{c(2\log2-\log3)}{40\log2-8\log3}}).$$
Since $R_2(A_0,n)=R_2(B_0,n)$, we have
$$R_2(A_0,n)=\frac{n}{8}+O(n^{1-\frac{c(2\log2-\log3)}{40\log2-8\log3}}).$$
Let $\epsilon=\frac{c(2\log2-\log3)}{40\log2-8\log3}$. Then $0<\epsilon<\frac{2\log2-\log3}{40\log2-8\log3}$. By Stirling's approximation, there is a $\delta'=\delta'(\epsilon)$ with $0<\delta'<1$ such that
\begin{eqnarray*}
&&|\{n:n\leq x,R_2(A_0,n)\neq\frac{n}{8}+O(n^{1-\epsilon})\}|\\
&\leq&|\{n:n\leq x,f(n)\leq\frac{c\log_2x}{24}\}|\\
&\leq&\sum_{i=0}^{\lfloor \frac{c\log_2x}{24}\rfloor}
\left(\begin{array}{c}
\lfloor\frac{\log_2x}{3}\rfloor+1 \\
i
\end{array}
\right)7^{\lfloor\frac{\log_2x+1}{3}\rfloor-i}\\
&\leq&\sum_{i=0}^{\lfloor \frac{c\log_2x}{24}\rfloor}
\left(\begin{array}{c}
\lfloor\frac{\log_2x}{3}\rfloor+1 \\
\lfloor \frac{c\log_2x}{24}\rfloor
\end{array}
\right)7^{\lfloor\frac{\log_2x+1}{3}\rfloor-\lfloor \frac{c\log_2x}{24}\rfloor}=O(x^{1-\delta'}).
\end{eqnarray*}
Therefore,
$$|\{n:\frac{x}{10}<n\leq x,R_2(A_0,n)=\frac{n}{8}+O(n^{1-\epsilon})\}|=\frac{9x}{10}+O(x^{1-\delta'}).$$
This completes the proof.
\end{proof}

\begin{lemma}\cite[Lemma 1]{CT}\label{lem2}
Let $A$ be a subset of $\mathbb{N}$ and $N$ be a positive integer such that $R_2(A,n)=R_2(\mathbb{N}\setminus A,n)$ for all $n\geq 2N-1$, and let $m,k,i$ be integers with $m\geq N$, $i\geq 0$ and $0\leq k<2^i$. Then

(a) if $k\in A_0$, then $m\in A\Leftrightarrow2^im+k\in A$,

(b) if $k\in B_0$, then $m\in A\Leftrightarrow2^im+k\notin A$.
\end{lemma}

\begin{lemma}\label{lem3}
Let $A$ be a subset of $\mathbb{N}$ and $N$ be a positive integer such that $R_2(A,n)=R_2(\mathbb{N}\setminus A,n)$ for all $n\geq 2N-1$. If $0<\epsilon,\delta'<1$ and
\begin{equation}\label{eq2.3}
|\{n:n\leq x,R_2(A_0,n)=\frac{n}{8}+O(n^{1-\epsilon})\}|=x+O(x^{1-\delta'}),
\end{equation}
then
$$|\{n:\frac{x}{10}\leq n\leq x,R_2(A,n)=\frac{n}{8}+O(n^{\frac{1}{1+\epsilon}})\}|=\frac{9x}{10}+O(x^{1-\frac{\delta'}{1+\epsilon}}).$$
\end{lemma}

\begin{proof}
Let $2^{k_0-2}<N\leq 2^{k_0-1}$ and let $k=\lfloor\frac{\log_2x}{1+\epsilon}\rfloor$. The integer $\frac{x}{10}<n\leq x$ is called bad if $n=2^km+r,~0\leq r<2^k$, where
$R_2(A_0,r)\neq \frac{r}{8}+O(2^{k(1-\epsilon)})$ or $R_2(A_0,r+2^k)\neq \frac{r+2^k}{8}+O(2^{k(1-\epsilon)})$. We are going to show that if $n$ is not bad, then
$$R_2(A,n)=\frac{n}{8}+O(n^{\frac{1}{1+\epsilon}}).$$

Suppose that $n$ is not bad. Then
$$R_2(A_0,r)=R_2(B_0,r)=\frac{r}{8}+O(2^{k(1-\epsilon)}),$$
and so
\begin{equation}\label{eq2.1}
R_{A_0,A_0}(r)=\frac{r}{4}+O(2^{k(1-\epsilon)}),~~R_{B_0,B_0}(r)=\frac{r}{4}+O(2^{k(1-\epsilon)}).
\end{equation}
It follows from
$$R_{A_0,A_0}(r)+R_{B_0,B_0}(r)+R_{A_0,B_0}(r)+R_{B_0,A_0}(r)=r+1$$
that
\begin{equation}\label{eq2.2}
R_{A_0,B_0}(r)=R_{B_0,A_0}(r)=\frac{r}{4}+O(2^{k(1-\epsilon)}).
\end{equation}
Similarly, we have
\begin{equation}\label{eq2.4}
R_{A_0,A_0}(r+2^k)=R_{B_0,B_0}(r+2^k)=\frac{r+2^k}{4}+O(2^{k(1-\epsilon)}).
\end{equation}
and
\begin{equation}\label{eq2.5}
R_{A_0,B_0}(r+2^k)=R_{B_0,A_0}(r+2^k)=\frac{r+2^k}{4}+O(2^{k(1-\epsilon)}).
\end{equation}

By Lemma \ref{lem2}, we know that
$$A\cap[0,x]=(A\cap[0,2^{k_0+k}-1])\cup\left(\cup_{j=2^{k_0}}^{\lfloor\frac{x}{2^k}\rfloor-1}(j2^k+C_j)\right)\cup(A\cap[\lfloor\frac{x}{2^k}\rfloor2^k,x]),$$
where $C_j=A_0\cap[0,2^k-1]$ if $j\in A$ and $C_j=B_0\cap[0,2^k-1]$ if $j\in \mathbb{N}\setminus A$. Define $\bar{C_j}=[0,2^k-1]\setminus C_j$, and $C_j^{ext}=A_0$ if $j\in A$, $C_j^{ext}=B_0$ if $j\in \mathbb{N}\setminus A$.
We have
\begin{eqnarray*}
R_2(A,n)&=&\sum\limits_{j=\lfloor\frac{n}{2^{k+1}}\rfloor+2}^{\lfloor\frac{n}{2^{k}}\rfloor-2^{k_0}}(|\{(a,a'):a=(\lfloor\frac{n}{2^{k}}\rfloor-j)2^k+c,a'=j2^k+c',0\leq c,c'<2^k,\\
&&a+a'=n,a,a'\in A\}|+|\{(a,a'):a=(\lfloor\frac{n}{2^{k}}\rfloor-j-1)2^k+c,a'=j2^k+c',\\
&&0\leq c,c'<2^k, a+a'=n,a,a'\in A\}|)+O(2^k).\\
&=&\sum\limits_{j=\lfloor\frac{n}{2^{k+1}}\rfloor+2}^{\lfloor\frac{n}{2^{k}}\rfloor-2^{k_0}}(|\{(c,c'):c\in C_{\lfloor\frac{n}{2^{k}}\rfloor-j},c'\in C_j,c+c'=n-\lfloor\frac{n}{2^{k}}\rfloor2^k\}|\\
&&+|\{(c,c'):c\in C_{\lfloor\frac{n}{2^{k}}\rfloor-j-1},c'\in C_j,c+c'=n-\lfloor\frac{n}{2^{k}}\rfloor2^k+2^k\}|)+O(2^k).
\end{eqnarray*}
By (\ref{eq2.1}) and (\ref{eq2.2}), we have
\begin{eqnarray*}
&&|\{(c,c'):c\in C_{\lfloor\frac{n}{2^{k}}\rfloor-j},c'\in C_j,c+c'=n-\lfloor\frac{n}{2^{k}}\rfloor2^k\}|\\
&=&R_{C_{\lfloor\frac{n}{2^{k}}\rfloor-j},C_j}(n-\lfloor\frac{n}{2^{k}}\rfloor2^k)=\frac{n-\lfloor\frac{n}{2^{k}}\rfloor2^k}{4}+O(2^{k(1-\epsilon)}).
\end{eqnarray*}
On the other hand,
\begin{eqnarray*}
&&|\{(c,c'):c\in C_{\lfloor\frac{n}{2^{k}}\rfloor-j-1},c'\in C_j,c+c'=n-\lfloor\frac{n}{2^{k}}\rfloor2^k+2^k\}|\\
&=&|\{(c,c'):c\in C_{\lfloor\frac{n}{2^{k}}\rfloor-j-1}^{ext},c'\in C_j^{ext},c+c'=n-\lfloor\frac{n}{2^{k}}\rfloor2^k+2^k\}|\\
&&-|\{(c,c'):c\geq 2^k,c\in C_{\lfloor\frac{n}{2^{k}}\rfloor-j-1}^{ext},c'\in C_j^{ext},c+c'=n-\lfloor\frac{n}{2^{k}}\rfloor2^k+2^k\}|\\
&&-|\{(c,c'):c'\geq2^k,c\in C_{\lfloor\frac{n}{2^{k}}\rfloor-j-1}^{ext},c'\in C_j^{ext},c+c'=n-\lfloor\frac{n}{2^{k}}\rfloor2^k+2^k\}|.
\end{eqnarray*}
By (\ref{eq2.4}) and (\ref{eq2.5}), we have
\begin{eqnarray*}
&&|\{(c,c'):c\in C_{\lfloor\frac{n}{2^{k}}\rfloor-j}^{ext},c'\in C_j^{ext},c+c'=n-\lfloor\frac{n}{2^{k}}\rfloor2^k+2^k\}|\\
&=&R_{C_{\lfloor\frac{n}{2^{k}}\rfloor-j}^{ext},C_j^{ext}}(n-\lfloor\frac{n}{2^{k}}\rfloor2^k+2^k)=\frac{n-\lfloor\frac{n}{2^{k}}\rfloor2^k+2^k}{4}+O(2^{k(1-\epsilon)}).
\end{eqnarray*}
If $c+c'=n-\lfloor\frac{n}{2^{k}}\rfloor2^k+2^k$, where $c\geq 2^k$, then $2^k\leq c<2^{k+1}$. Hence, $c\in C_{\lfloor\frac{n}{2^{k}}\rfloor-j-1}^{ext}$ if and only if $c-2^k\in \bar{C}_{\lfloor\frac{n}{2^{k}}\rfloor-j-1}$. Therefore,
\begin{eqnarray*}
&&|\{(c,c'):c\geq 2^k,c\in C_{\lfloor\frac{n}{2^{k}}\rfloor-j-1}^{ext},c'\in C_j^{ext},c+c'=n-\lfloor\frac{n}{2^{k}}\rfloor2^k+2^k\}|\\
&=&|\{(c,c'):c\in \bar{C}_{\lfloor\frac{n}{2^{k}}\rfloor-j-1},c'\in C_j,c+c'=n-\lfloor\frac{n}{2^{k}}\rfloor2^k\}|\\
&=&R_{\bar{C}_{\lfloor\frac{n}{2^{k}}\rfloor-j-1},C_j}(n-\lfloor\frac{n}{2^{k}}\rfloor2^k)=\frac{n-\lfloor\frac{n}{2^{k}}\rfloor2^k}{4}+O(2^{k(1-\epsilon)}).
\end{eqnarray*}
Similarly,
\begin{eqnarray*}
&&|\{(c,c'):c'\geq 2^k,c\in C_{\lfloor\frac{n}{2^{k}}\rfloor-j-1}^{ext},c'\in C_j^{ext},c+c'=n-\lfloor\frac{n}{2^{k}}\rfloor2^k+2^k\}|\\
&=&|\{(c,c'):c\in C_{\lfloor\frac{n}{2^{k}}\rfloor-j-1},c'\in \bar{C}_j,c+c'=n-\lfloor\frac{n}{2^{k}}\rfloor2^k\}|\\
&=&R_{C_{\lfloor\frac{n}{2^{k}}\rfloor-j-1},\bar{C}_j}(n-\lfloor\frac{n}{2^{k}}\rfloor2^k)=\frac{n-\lfloor\frac{n}{2^{k}}\rfloor2^k}{4}+O(2^{k(1-\epsilon)}).
\end{eqnarray*}
It follows that
\begin{eqnarray*}
|\{(c,c'):c\in C_{\lfloor\frac{n}{2^{k}}\rfloor-j-1},c'\in C_j,c+c'=n-\lfloor\frac{n}{2^{k}}\rfloor2^k+2^k\}|
=\frac{2^k}{4}-\frac{n-\lfloor\frac{n}{2^{k}}\rfloor2^k}{4}+O(2^{k(1-\epsilon)}).
\end{eqnarray*}
Hence,
\begin{eqnarray*}
R_2(A,n)=\sum_{j=\lfloor\frac{n}{2^{k+1}}\rfloor+2}^{\lfloor\frac{n}{2^{k}}\rfloor-2^{k_0}}\left(\frac{2^k}{4}+O(2^{k(1-\epsilon)})\right)+O(2^k)=\frac{n}{8}+O(n^{\frac{1}{1+\epsilon}}).
\end{eqnarray*}
By (\ref{eq2.3}), we have
\begin{eqnarray*}
|\{n:n\leq x,R_2(A,n)\neq\frac{n}{8}+O(n^{\frac{1}{1+\epsilon}})\}|
\leq |\{n:n\leq x, n~\text{is~bad}\}|
=O\left(2^{k(1-\delta')}\frac{x}{2^k}\right)=O(x^{1-\frac{\delta'}{1+\epsilon}}).
\end{eqnarray*}
Therefore,
$$|\{n:\frac{x}{10}\leq n\leq x,R_2(A,n)=\frac{n}{8}+O(n^{\frac{1}{1+\epsilon}})\}|=\frac{9x}{10}+O(x^{1-\frac{\delta'}{1+\epsilon}}).$$
This completes the proof.
\end{proof}

Finally, we prove our main results.
\begin{proof}[Proof of Theorems \ref{thm0} and \ref{thm1}]
For $0<\epsilon<\frac{2\log2-\log3}{40\log 2-8\log3}$, by Lemma \ref{lem1}, we know that there exists a $\delta'=\delta'(\epsilon)$ with $0<\delta'<1$ such that
\begin{eqnarray*}
&&|\{n:n\leq x,R_2(A_0,n)=\frac{n}{8}+O(n^{1-\epsilon})\}|\\
&=&\sum_{i=0}^{\delta'\log_{10}x}|\{n:\frac{x}{10^{i+1}}\leq n\leq \frac{x}{10^i},R_2(A_0,n)=\frac{n}{8}+O(n^{1-\epsilon})\}|+O\left(\frac{x}{10^{\delta'\log_{10}x}}\right)\\
&=&\sum_{i=0}^{\delta'\log_{10}x}\left(\frac{9x}{10^{i+1}}+O\left(\left(\frac{x}{10^i}\right)^{1-\delta'}\right)\right)+O(x^{1-\delta'})\\
&=&x+O(x^{1-\delta'}).
\end{eqnarray*}
Let $\theta=\frac{\epsilon}{1+\epsilon}$. Then $0<\theta<\frac{2\log2-\log3}{42\log 2-9\log3}$. By Lemma \ref{lem3}, we have
$$|\{n:\frac{x}{10}\leq n\leq x,R_2(A,n)=\frac{n}{8}+O(n^{1-\theta})\}|=\frac{9x}{10}+O(x^{1-\frac{\delta'}{1+\epsilon}}).$$
Similar to the previous argument, we can get the desired result.
\end{proof}

\section*{Acknowledgements}
The first author was supported by the National Natural Science Foundation of China, Grant No. 11771211. The second author was supported by the NKFIH Grant No. K129335. The third author was supported by the National Natural Science Foundation for Youth of China, Grant No. 11501299, the Natural Science Foundation of Jiangsu Province, Grant Nos. BK20150889, 15KJB110014 and the Startup Foundation for Introducing Talent of NUIST, Grant No. 2014r029.

\renewcommand{\refname}{Bibliography}

\end{document}